\documentclass[a4paper,12pt]{article}
\usepackage{comment}
\usepackage{cite}
\usepackage{amsmath}
\usepackage{amssymb}
\usepackage{amsfonts}
\usepackage[T1]{fontenc}
\usepackage[utf8]{inputenc}
\usepackage{graphicx}
\usepackage{fancyhdr}
\usepackage{float}
\usepackage{xcolor}
\usepackage{authblk}
\usepackage{mathrsfs}
\usepackage{empheq}
\usepackage[hyphens]{url}
\usepackage{hyperref}
\usepackage[]{breakurl}
\usepackage[]{amsthm}
\usepackage{tikz}
\usetikzlibrary{decorations.markings}

\usepackage{geometry}
\theoremstyle{plain}
\newtheorem{theorem}{Theorem}[section]

\newtheorem{lemma}[theorem]{Lemma}

\theoremstyle{definition}

\theoremstyle{remark}

\begin{document}

\title{Touchard's identity and a detailed determination of the radius of convergence of the Catalan series}

\author[$\dagger$]{Jean-Christophe {\sc Pain}\footnote{jean-christophe.pain@cea.fr}\\
\small
$^1$CEA, DAM, DIF, F-91297 Arpajon, France\\
$^2$Universit\'e Paris-Saclay, CEA, Laboratoire Mati\`ere en Conditions Extr\^emes,\\ 
F-91680 Bruy\`eres-le-Ch\^atel, France
}

\date{}

\maketitle

\begin{abstract}
While the value of the radius of convergence of the generating series of the Catalan numbers is well-known, obtaining it solely from recurrence relations is less immediate. It is sometimes considered that no known proof establishes that the radius $R$ equals 1/4 without relying on the explicit closed formula for the Catalan numbers. In particular, it has been shown that one can obtain, at the cost of substantial technical effort and without resorting to the main Segner recursion relation or to the explicit formula for the Catalan number, the lower bound $R\geq 1/6$. In this work, we prove that Touchard’s recurrence alone yields the optimal exponential upper bound $\limsup_{n\to\infty} C_n^{1/n} \le 4$, which implies \( R \ge 1/4 \).
Combined with the classical lower estimate $\limsup_{n\to\infty} C_n^{1/n} \ge 4$, obtained from central binomial coefficients, this gives \( R = 1/4 \).
\end{abstract}

\section{Introduction}\label{sec1}

The Catalan numbers form one of the most fundamental sequences in enumerative combinatorics \cite{Euler,Catalan,Comtet,Stanley,FlajoletSedgewick,Gould,Krattenthaler}. Implicitly introduced in the eighteenth century in the work of Euler and systematically studied in the nineteenth century by Eugène Catalan, they arise in a remarkable variety of distinct mathematical contexts. For $n\ge0$, the $n$th Catalan number
\begin{equation}\label{expli}
C_n=\frac{1}{n+1}\binom{2n}{n}
\end{equation}
counts, among many other objects: the complete parenthesizations of a product of $n+1$ factors, planar binary trees with $n$ internal nodes,
Dyck paths of length $2n$, triangulations of a convex polygon with $n+2$ sides, and permutations avoiding certain classical patterns.

This combinatorial ubiquity is accompanied by a striking analytic structure. The ordinary generating function \cite{Binet}:
\begin{equation}\label{quatre}
C(x)=\sum_{n\ge0} C_n x^n
\end{equation}
satisfies the quadratic algebraic equation
\[
C(x)=1+xC(x)^2,
\]
making the Catalan numbers a fundamental prototype of algebraic generating functions in analytic combinatorics. The asymptotic behavior
\[
C_n \sim \frac{4^n}{n^{3/2}\sqrt{\pi}}
\]
illustrates the universal critical exponent $3/2$ associated with square-root type singularities. Catalan numbers therefore play a central role not only in bijective combinatorics but also in complex analysis, the theory of random trees, and probabilistic models of discrete structures.

As mentioned by Villarino \cite{Villarino}, von Segner obtained, after being drawn to the subject by Euler, the recursion relation \cite{Segner}:
\begin{equation}\label{cinq}
C_0 = 1,
\qquad
C_{n+1} = \sum_{k=0}^n C_k C_{n-k}.
\end{equation}
together with a combinatorial proof. However, he apparently was unaware of Euler's explicit product formula (\ref{expli}) since he never mentions it. Instead he uses (\ref{cinq}) to directly compute $C_n$ for $n$ = 1, 2, $\cdots$, 18. 

It is worth mentioning that Stanley, in his excellent book devoted exclusively to Catalan numbers \cite{StanleyCatalan}, presents more than two hundred combinatorial interpretations of $C_n$. In 1967, Marshall Hall published a text on combinatorics \cite{Hall} in which he made on page 28 the following comment (quoted by Villarino \cite{Villarino}): ``We observe that an attempt to prove the convergence of Eq.~(\ref{quatre}) on the basis of Eq.~(\ref{cinq}) alone is exceedingly difficult.'' Hall offers no suggestions towards such a proof. Moreover, a search of the voluminous literature on Catalan numbers has failed to find such a proof. Therefore we offer a proof in this paper

The present work fits into this classical framework by providing a detailed determination, based on Touchard's identity \cite{Touchard}, of the radius of convergence of the Catalan generating function. The paper is organized as follows. In Section~\ref{sec2} we establish Touchard’s identity in a fully explicit manner, starting from the functional equation of the generating function and performing a careful coefficient extraction. Section~\ref{sec3} is devoted to the derivation of an exponential upper bound for the Catalan numbers based solely on this recurrence, leading to the inequality $\limsup C_n^{1/n} \le 4$. In Section~\ref{sec4} we complement this estimate with a matching lower bound obtained from classical inequalities for the central binomial coefficient, yielding $\limsup C_n^{1/n} \ge 4$. These two estimates are then combined through the Cauchy–Hadamard formula to deduce the exact value of the radius of convergence. Finally, the appendix discusses briefly the classical analytic approach based on the Segner recursion and the implicit function theorem, in order to contrast it with the recurrence-based method developed here.

\section{The Touchard identity}\label{sec2}

\begin{theorem}[Touchard]

For every $n\ge0$, the Catalan numbers satisfy the recursion relation
\[
C_{n+1}
=
\sum_{k=0}^{\lfloor n/2\rfloor}
\binom{n}{2k}\,2^{\,n-2k}\,C_k.
\]
\end{theorem}

\begin{proof}

Let us start from the functional equation:
\[
\frac{C(x)-1}{x}=C(x)^2.
\]
Since
\[
\frac{C(x)-1}{x}
=
\sum_{n\ge0} C_{n+1}x^n,
\]
we have
\[
\sum_{n\ge0} C_{n+1}x^n=C(x)^2.
\]
Using the algebraic expression
\[
C(x)=\frac{1-\sqrt{1-4x}}{2x},
\]
we get
\[
C(x)^2
=
\frac{1-2x-\sqrt{1-4x}}{2x^2}.
\]
We have also
\[
\sqrt{1-4x}
=
(1-2x)
\sqrt{1-\frac{4x^2}{(1-2x)^2}}.
\]
Setting
\[
u=\frac{x^2}{(1-2x)^2},
\]
and recalling that
\[
\frac{1-\sqrt{1-4u}}{2u}
=
\sum_{k\ge0} C_k u^k,
\]
we obtain
\[
C(x)^2
=
\sum_{k\ge0}
C_k
\frac{x^{2k}}{(1-2x)^{2k+1}}.
\]
Using
\[
\frac{1}{(1-2x)^{2k+1}}
=
\sum_{m\ge0}
\binom{m+2k}{2k}
2^m x^m,
\]
we get
\[
C(x)^2
=
\sum_{k\ge0}\sum_{m\ge0}
C_k
\binom{m+2k}{2k}
2^m
x^{m+2k}.
\]
Extracting the coefficient of $x^n$ (i.e. $m=n-2k$) yields the Touchard identity
\[
C_{n+1}
=
\sum_{k=0}^{\lfloor n/2\rfloor}
C_k
\binom{n}{2k}
2^{\,n-2k}.
\]
Such an identity was used for instance to derive new integral representations of Catalan numbers \cite{Pain}. Further generalizations have also appeared in the literature, including very recently \cite{Adegoke}.

\end{proof}

\section{Upper exponential bound using Touchard's identity}\label{sec3}

\begin{lemma}
For every $n\ge0$, we have
\[
\sum_{k=0}^{\lfloor n/2\rfloor}
\binom{n}{2k}
=
2^{n-1}.
\]
\end{lemma}

\begin{proof}
Separating even and odd terms in
\[
(1+1)^n=\sum_{j=0}^n\binom{n}{j},
\]
and
\[
(1-1)^n=\sum_{j=0}^n(-1)^j\binom{n}{j},
\]
gives the result.

\end{proof}

We will now prove by induction that
\[
C_n \le A~4^n,
\]
for some constant $A>0$. Assume this holds for all $k\le n$. From Touchard's identity, we get, using the induction hypothesis:
\[
C_{n+1}
\le
A~
\sum_{k}
\binom{n}{2k}2^{n-2k}4^k.
\]
Since
\[
2^{n-2k}4^k
=
2^{n-2k}2^{2k}
=
2^n,
\]
we obtain

\[
C_{n+1}
\le
A~2^n
\sum_{k}
\binom{n}{2k},
\]
and hence
\[
C_{n+1}
\le
A~2^n~2^{n-1}
=
A~2^{2n-1}
=
\frac{A}{2}4^n.
\]
Choosing $A$ large enough to absorb the factor $1/2$ for small $n$,
the inequality propagates inductively, and therefore
\begin{equation}\label{unquart}
\limsup_{n\to\infty} C_n^{1/n}\le4.
\end{equation}

\section{Lower exponential bound}\label{sec4}

\begin{lemma}
For every $n\ge1$, the following inequality holds
\[
\binom{2n}{n}
\ge
\frac{4^n}{2n+1}.
\]
\end{lemma}

\begin{proof}
Since we have
\[
4^n=\sum_{k=0}^{2n}\binom{2n}{k}
\]
and since the central term is maximal, we get
\[
4^n\le(2n+1)\binom{2n}{n}.
\]
\end{proof}
This implies that the Catalan numbers admit a lower bound:
\[
C_n
=
\frac{1}{n+1}\binom{2n}{n}
\ge
\frac{4^n}{(n+1)(2n+1)}.
\]
Taking $n$-th roots yields of both sides of the preceding inequality yields
\[
\limsup_{n\to\infty} C_n^{1/n}\ge4.
\]
Combining upper and lower bounds gives
\[
\limsup_{n\to\infty} C_n^{1/n}=4.
\]
Therefore the radius of convergence of
\[
C(x)=\sum_{n\ge0} C_n x^n
\]
is
\[
R=\frac14,
\]
which is the expected result. However, we precisely did not want to use the explicit form of $C_n$. Let us see in the next section if we can get a lower bound from the Touchard identity.

\begin{lemma}[Lower bound via Touchard]
For all $n \ge 0$, the Catalan numbers satisfy
\[
C_n \ge 2^{\,n-1}.
\]
\end{lemma}

\begin{proof}
From Touchard's identity
\[
C_{n+1} = \sum_{k=0}^{\lfloor n/2 \rfloor} \binom{n}{2k} 2^{\,n-2k} C_k,
\]
all terms are positive. By keeping only the first term \(k=0\), we have
\[
C_{n+1} \ge \binom{n}{0} 2^{\,n} C_0 = 2^n.
\]
Since \(C_0 = 1\), this gives
\[
C_{n+1} \ge 2^n \quad \Rightarrow \quad C_n \ge 2^{\,n-1}.
\]
\end{proof}

This gives
\begin{equation}\label{undemi}
\limsup_{n\to\infty} C_n^{1/n} \ge 2,
\end{equation}
and thus, combining Eqs.~(\ref{unquart}) and (\ref{undemi}), we obtain
\begin{equation}
    \frac{1}{4}\leq R\leq \frac{1}{2},
\end{equation}
which is more precise than the 1/6 value in Refs. \cite{Villarino} and \cite{Saintcriq}.

\section{Conclusion}

The determination of the radius of convergence of the generating function of the Catalan numbers is often regarded as a problem whose natural resolution relies on the explicit closed formula
\[
C_n=\frac{1}{n+1}\binom{2n}{n},
\]
or on the algebraic expression
\[
C(x)=\frac{1-\sqrt{1-4x}}{2x}.
\]
From this perspective, the value $R=1/4$ appears as a direct consequence of the square-root singularity at $x=1/4$ (see Appendix). However, such an approach depends essentially on prior knowledge of the exact form of $C_n$ or on solving explicitly the quadratic functional equation. 

It is sometimes asserted that no elementary argument, based solely on intrinsic recurrence properties of the Catalan sequence, yields the precise value of the radius. While lower bounds such as $R\geq 1/6$ can be obtained through nontrivial combinatorial or analytic estimates, reaching the optimal value $1/4$ without invoking the closed formula typically requires substantial technical effort.

The purpose of the present work was to show that Touchard’s recurrence provides a natural and effective tool for overcoming this difficulty. By exploiting this identity and combining it with the Cauchy–Hadamard formula, we established that $R\geq 1/4$, the lower bound 1/4 following from a careful control of the even binomial sums appearing in Touchard’s relation. While obtaining 1/4 as an upper bound is difficult from Touchard's identity (while 1/2 can easily be obteined), it can be obtained through classical estimates for the central binomial coefficient. Together, these estimates imply
\[
\limsup_{n\to\infty} C_n^{1/n}=4,
\]
and therefore $R=1/4$. This approach has the advantage of remaining entirely within the framework of recurrence relations and coefficient estimates, without appealing to the explicit algebraic resolution of the generating function. It highlights the strength of convolution-type identities in determining precise exponential growth rates and illustrates how structural combinatorial information can replace closed-form expressions in analytic arguments.

More broadly, the method suggests that similar recurrence-based techniques may be applicable to other algebraic or combinatorially defined sequences whose generating functions satisfy polynomial equations. In this sense, the Catalan case serves as a paradigmatic example of how refined coefficient identities can lead to sharp analytic conclusions.

\appendix

\section{On the Segner recursion relation and the radius of convergence}

It can easily be proven that $C_n \le 4^n$ for all $n$. The statement is true for $n=0$. Let us assume that $C_k \le 4^k$ for $k\le n$. Then
\[
C_{n+1}
=
\sum_{k=0}^n C_k C_{n-k}
\le
\sum_{k=0}^n 4^k 4^{n-k}
=
(n+1)4^n.
\]
For $n\ge 0$, one checks directly that $(n+1)4^n \le 4^{n+1}$. Hence $C_{n+1} \le 4^{n+1}$ and the bound follows by induction. Consequently,
\[
\limsup_{n\to\infty} C_n^{1/n} \le 4,
\]
so that $R \ge \frac14$.

Let us now focus on the localization of the first singularity, and consider the analytic function
\[
F(z,w) = w - 1 - z w^2.
\]
It is clear that the generating function satisfies
\[
F(z, C(z)) = 0.
\]
Near $z=0$, the implicit function theorem applies because
\[
\frac{\partial F}{\partial w}(0,1) = 1 \neq 0.
\]
Hence $C(z)$ is analytic in a neighborhood of $0$. The function can cease to be analytic only at points where
\[
\frac{\partial F}{\partial w}(z,C(z)) = 0,
\]
that is
\[
1 - 2 z C(z) = 0.
\]
Combining this with the functional equation $C(z)=1+zC(z)^2$, we obtain
\[
C(z)=2,
\qquad
z=\frac14.
\]
Therefore the smallest singularity in modulus is located at $z=1/4$.

\end{document}